\documentclass[12pt]{amsart}
\usepackage{amsmath}
\usepackage{amsthm}
\usepackage{amsfonts}
\usepackage{amssymb}
\usepackage{mathrsfs}
\usepackage{stmaryrd}
\usepackage[shortlabels]{enumitem}
\usepackage{xcolor}
\usepackage[all]{xy}
\usepackage[margin=1in]{geometry} 
\usepackage[bookmarks, bookmarksdepth=2, colorlinks=true, linkcolor=blue, citecolor=blue, urlcolor=blue]{hyperref}
\usepackage{eucal}
\usepackage{contour}
\usepackage[normalem]{ulem}
\usepackage{tikz}
\usetikzlibrary{shapes.geometric}

\numberwithin{equation}{section}
\newtheorem{theorem}[equation]{Theorem}

\newtheorem{proposition}[equation]{Proposition}
\newtheorem{lemma}[equation]{Lemma}
\newtheorem{corollary}[equation]{Corollary}

\theoremstyle{definition}
\newtheorem{rmk}[equation]{Remark}
\newenvironment{remark}[1][]{\begin{rmk}[#1] \pushQED{\qed}}{\popQED \end{rmk}}
\newtheorem{eg}[equation]{Example}
\newenvironment{example}[1][]{\begin{eg}[#1] \pushQED{\qed}}{\popQED \end{eg}}
\newtheorem{defnaux}[equation]{Definition}
\newenvironment{definition}[1][]{\begin{defnaux}[#1]\pushQED{\qed}}{\popQED \end{defnaux}}

\newcommand{\bA}{\mathbf{A}}

\newcommand{\cC}{\mathcal{C}}

\newcommand{\fI}{\mathfrak{I}}

\newcommand{\bN}{\mathbf{N}}

\newcommand{\cO}{\mathcal{O}}

\newcommand{\bS}{\mathbf{S}}

\newcommand{\fS}{\mathfrak{S}}

\newcommand{\cU}{\mathcal{U}}

\newcommand{\bV}{\mathbf{V}}

\newcommand{\cZ}{\mathcal{Z}}

\newcommand{\fa}{\mathfrak{a}}

\newcommand{\fb}{\mathfrak{b}}

\newcommand{\rf}{\mathrm{f}}

\contourlength{1pt}

\contourlength{0.8pt}
\newcommand{\myuline}[1]{%
  \uline{\phantom{#1}}%
  \llap{\contour{white}{#1}}%
}
\DeclareMathOperator{\uRep}{\text{\myuline{\rm Rep}}}
\DeclareMathOperator{\uage}{\text{\myuline{\rm age}}}
\DeclareMathOperator{\umu}{\text{\myuline{\rm $\mu$}}}

\let\ol\overline
\let\ul\underline
\DeclareMathOperator{\im}{im}

\DeclareMathOperator{\Sym}{Sym}
\DeclareMathOperator{\Aut}{Aut}
\DeclareMathOperator{\Spec}{Spec}

\DeclareMathOperator{\Rep}{Rep}
\DeclareMathOperator{\Fun}{Fun}
\DeclareMathOperator{\Ind}{Ind}
\renewcommand{\phi}{\varphi}
\renewcommand{\emptyset}{\varnothing}

\newcommand{\pol}{\mathrm{pol}}
\renewcommand{\Vec}{\mathrm{Vec}}
\newcommand{\GL}{\mathbf{GL}}

\DeclareMathOperator{\age}{age}
\DeclareMathOperator{\Sh}{Sh}
\newcommand{\Mod}{\mathrm{Mod}}
\newcommand{\unu}{\ul{\smash{\nu}}}
\newcommand{\Pol}{\mathbf{Pol}}
\newcommand{\ulambda}{\ul{\smash{\lambda}}}

\let\defn\emph
\newcommand{\DOI}[1]{\href{http://doi.org/#1}{\color{purple}{\tiny\tt DOI:#1}}}
\newcommand{\arxiv}[1]{\href{http://arxiv.org/abs/#1}{{\tiny\tt arXiv:#1}}}

\author{Nate Harman}
\address{Department of Mathematics, University of Georgia, Athens, GA}
\email{\href{mailto:nharman@uga.edu}{nharman@uga.edu}}
\urladdr{\url{https://www.nateharman.com/}}

\author{Andrew Snowden}
\address{Department of Mathematics, University of Michigan, Ann Arbor, MI}
\email{\href{mailto:asnowden@umich.edu}{asnowden@umich.edu}}
\urladdr{\url{http://www-personal.umich.edu/~asnowden/}}
\thanks{AS was supported by NSF DMS-2301871.}

\title[Tensor spaces and the geometry of polynomial representations]{Tensor spaces and the geometry of\\ polynomial representations}
\date{July 25, 2024}

\begin{document}

\begin{abstract}
A \defn{tensor space} is a vector space equipped with a finite collection of multi-linear forms. In previous work, we showed that (for each signature) there exists a universal homogeneous tensor space, which is unique up to isomorphism. Here we generalize that result: we show that each Zariski class of tensor spaces contains a weakly homogeneous space, which is unique up to isomorphism; here, we say that two tensor spaces are \defn{Zariski equivalent} if they satisfy the same polynomial identities. Our work relies on the theory of $\GL$-varieties developed by Bik, Draisma, Eggermont, and Snowden.
\end{abstract}

\maketitle
\tableofcontents

\section{Introduction}

\subsection{Tensor spaces} \label{ss:bg}

Fix an algebraically closed field $k$ of characteristic~0. For a tuple $\ulambda=[\lambda_1, \ldots, \lambda_r]$ of partitions, we define a \defn{$\ulambda$-space} to be a $k$-vector space $V$ equipped with linear maps $\omega_i \colon \bS_{\lambda_i}(V) \to k$ for each $1 \le i \le r$, where $\bS_{\lambda}$ denotes the usual Schur functor. For example, if $\ulambda=[(2)]$ then a $\ulambda$-space is a quadratic space, i.e., a vector space equipped with a symmetric bilinear form. The theory of countable $\ulambda$-spaces (i.e., those of dimension $\aleph_0$) has recently received some attention in the literature. We discuss three distinct threads of work, each of which ties in to the results of this paper.

\textit{(a) Universality.} A $\ulambda$-space is called \defn{universal} if every finite dimensional $\ulambda$-space embeds into it. Kazhdan and Ziegler \cite{KaZ2}, and later Bik, Danelon, Draisma and Eggermont \cite{BDDE}, established an important characterization of universal spaces (Theorem~\ref{thm:univ}). This line of work is intimately connected to strength (or Schmidt rank), which has proven so important in additive combinatorics \cite{GT,KaZ3} and commutative algebra \cite{AH, stillman}.

\textit{(b) Homogeneity.} In \cite{homoten}, we proposed that $\ulambda$-spaces can be viewed as linear analogs of the relational structures studied in model theory, and we transferred several important ideas from model theory to $\ulambda$-spaces. In particular, we defined a $\ulambda$-space $V$ to be \defn{homogeneous} if whenever $\alpha, \beta \colon W \to V$ are two embeddings of a finite dimensional $\ulambda$-space $W$, there is an automorphism $\sigma$ of $V$ such that $\beta=\sigma \circ \alpha$. One of our main results is that if $\ulambda$ is \defn{pure} (meaning no $\lambda_i$ is the empty partition) then a universal homogeneous $\ulambda$-space of countable dimension exists, and any two are isomorphic. It is quite remarkable that there is a distinguished isomorphism class of $\ulambda$-spaces.

\textit{(c) Isogeny.} Classifying $\ulambda$-spaces up to isomorphism is a hopeless problem. Two $\ulambda$-spaces are \defn{isogenous} if each embeds into the other. This is a coarser equivalence relation than isomorphism, and the evidence so far points to a good theory in this direction. In \cite{BDDE}, it is shown that there is a unique maximal isogeny class if $\ulambda$ is pure, and, in certain cases, a unique minimal universal isogeny class. In \cite{isocubic}, the isogeny classes of universal cubic spaces (i.e., $\ulambda$-spaces with $\ulambda=[(3)]$) are completely determined, with a quite elegant answer. A more comprehensive theory is forthcoming \cite{isoten}.

In this paper, we start a new thread: we examine $\ulambda$-spaces from the vantage of algebraic geometry, specifically, the theory of $\GL$-varieties developed by Bik, Draisma, Eggermont, and Snowden \cite{polygeom}. We apply results about these varieties to prove some new theorems about $\ulambda$-spaces. 

\subsection{$\GL$-varieties}

Before stating our results, we briefly recall the basic definitions surrounding $\GL$-varieties. Fix a $k$-vector space $\bV$ of countable infinite dimension, and let $\GL$ be its automorphism group. Let $\bA^{\ulambda}$ be the spectrum of the ring $R_{\ulambda} = \Sym(\bigoplus_{i=1}^r \bS_{\lambda_i}(\bV))$. An affine \defn{$\GL$-variety} is a reduced affine scheme over $k$ with an action of $\GL$ that admits an equivariant closed embedding into some $\bA^{\ulambda}$. The relevance to tensor spaces is immediate: $\bA^{\ulambda}$ is the parameter space for $\ulambda$-structures on $\bV$.

Classifying the $\GL$-orbits on a $\GL$-variety is like classifying $\ulambda$-spaces up to isomorphism, and is too difficult. Instead, following \cite{polygeom}, we work with the notion of \defn{generalized orbits}: two points belong to the same generalized orbit if each belongs to the orbit closure of the other. There it was shown that understanding generalized orbits is, in a sense, tractable: the space of generalized orbits is essentially built from combinatorial data and finite dimensional algebraic varieties.

We say that two countable $\ulambda$-spaces are \defn{Zariski equivalent} if the corresponding points of $\bA^{\ulambda}$ belong to the same generalized orbit; see \S \ref{ss:tenorb} for the general definition. Zariski equivalence is coarser than isogeny. We note that the universal $\ulambda$-spaces form a single Zariski class (Corollary~\ref{cor:univ-zar}); this Zariski class typically contains infinitely many isogeny classes.

\subsection{Results}

The main point of this paper is to extend prior results about the universal Zariski class to arbitrary Zariski classes. A key idea is that we can relax the notion of homogeneous tensor space: we define a \defn{weakly homogeneous tensor space} to be one that satisfies the condition in \S \ref{ss:bg}(b) for generic spaces $W$; see Definition~\ref{defn:weak-homo} for the precise condition. A countable quadratic space of finite rank at least two is weakly homogeneous (Example~\ref{ex:wh-fin-rk}), but not homogeneous (Example~\ref{ex:quad-homo}).

Our main theorem shows why weak homogeneity is a useful idea:

\begin{theorem} \label{mainthm}
In each Zariski class of $\ulambda$-space, there is a countable weakly homogenous space, which is unique up to isomorphism.
\end{theorem}

This theorem can also be phrased as follows:

\begin{corollary}
There is a natural bijection between generalized orbits on $\bA^{\ulambda}$ and isomorphism classes of countable weakly homogeneous $\ulambda$-spaces.
\end{corollary}

The proof of this theorem uses the universality theorem of \cite{KaZ2, BDDE}, the main structural results on $\GL$-varieties from \cite{polygeom}, and the construction of universal homogeneous $\ulambda$-spaces from \cite{homoten}. As a corollary, we greatly generalize the theorem of \cite{BDDE} on maximal isogeny classes:

\begin{corollary}
In each Zariski class, there is a unique maximal isogeny class.
\end{corollary}

In model theory, oligomorphic permutation groups play an important role. In \cite{homoten}, we introduced the notion of linear-oligomorphic group as the linear analog of oligomorphic group (see Definition~\ref{defn:lin-olig}), and we showed that the automorphism group of the countable universal homogeneous $\ulambda$-space is linear-oligomorphic. The following theorem generalizes this result. We discuss some motivation for this circle of ideas in \S \ref{ss:motiv}(c).

\begin{theorem}
If $V$ is a countable weakly homogeneous $\ulambda$-space then $\Aut(V)$ is a linear-oligomorphic group.
\end{theorem}

\subsection{Ages}

An important idea in model theory, introduced by Fra\"iss\'e, is the notion of the \defn{age} of a relational structure $X$: it is the class of all finite structures that embed into $X$. By analogy, we define the \defn{age} of a $\ulambda$-space $V$ to be the class of all finite dimensional $\lambda$-spaces that embed into $V$. We prove the following theorem:

\begin{theorem}
Two $\ulambda$-spaces $V$ and $W$ are Zariski equivalent if and only if they have the same age.
\end{theorem}

This is a significant result since it shows that the notion of Zariski equivalence has a direct interpretation in terms of model theory. In particular, the central idea of this paper---applying algebraic geometry to tensor spaces---is, in a sense, necessary to fully understand the model theory of tensor spaces.

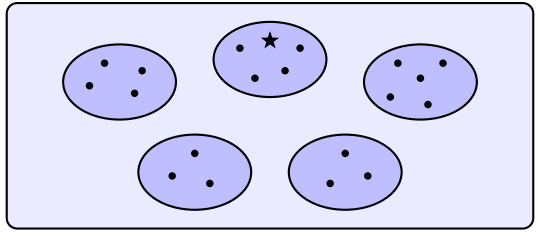
\begin{figure}
\begin{tikzpicture}
\draw[rounded corners, fill=blue!8, thick] (-3.5,1.5) rectangle (3.5,-1.5);
\draw[fill=blue!25, thick] (0, .75) ellipse (.75 and .5);
\draw[fill=blue!25, thick] (-2, .45) ellipse (.75 and .5);
\draw[fill=blue!25, thick] (-1,-.75) ellipse (.75 and .5);
\draw[fill=blue!25, thick] (1,-.75) ellipse (.75 and .5);
\draw[fill=blue!25, thick] (2, .45) ellipse (.75 and .5);
\node at (-2.2,.7)[circle,fill,inner sep=1pt]{};
\node at (-2.4,.4)[circle,fill,inner sep=1pt]{};
\node at (-1.8,.3)[circle,fill,inner sep=1pt]{};
\node at (-1.7,.6)[circle,fill,inner sep=1pt]{};
\node at (.4,.9)[circle,fill,inner sep=1pt]{};
\node at (.2,.6)[circle,fill,inner sep=1pt]{};
\node at (-.2,.5)[circle,fill,inner sep=1pt]{};
\node at (-.4,.9)[circle,fill,inner sep=1pt]{};
\node[star,star point ratio=2.5,minimum size=6pt,
          inner sep=0pt,draw=black,solid,fill=black] at (0,1) {};
\node at (2,.5)[circle,fill,inner sep=1pt]{};
\node at (2.1,.15)[circle,fill,inner sep=1pt]{};
\node at (1.6,.25)[circle,fill,inner sep=1pt]{};
\node at (1.7,.7)[circle,fill,inner sep=1pt]{};
\node at (2.3,.7)[circle,fill,inner sep=1pt]{};
\node at (-1.3,-.8)[circle,fill,inner sep=1pt]{};
\node at (-1,-.5)[circle,fill,inner sep=1pt]{};
\node at (-0.8,-.9)[circle,fill,inner sep=1pt]{};
\node at (1.3,-.8)[circle,fill,inner sep=1pt]{};
\node at (1,-.5)[circle,fill,inner sep=1pt]{};
\node at (0.8,-.9)[circle,fill,inner sep=1pt]{};
\end{tikzpicture}
\caption{A schematic picture of a Zariski class. Each oval is an isogeny class, and each point is an isomorphism class. The star is the isomorphism class of the weakly homogeneous space, and the oval to which it belongs is the maximal isogeny class. Typically, there are infinitely many isogeny classes, and infinitely many isomorphism classes in each isogeny class.}
\end{figure}

\subsection{Potential applications} \label{ss:motiv}

We now discuss a few potential applications of our results.

\textit{(a) Maps of $\GL$-varieties.} A standard approach to studying a map $\phi \colon Y \to X$ of algebraic varieties is to consider the fibers. Suppose now that $\phi$ is a map of $\GL$-varieties. We then run into a problem: the fiber $\phi^{-1}(x)$ over a point $x \in X$ does not carry an action of $\GL$ (in general), but only the stabilizer of $x$ in $\GL$. This stabilizer is often very small, or even trivial. Thus $\phi^{-1}(x)$ is typically an infinite dimensional variety with no special symmetry, and so it will be difficult to say anything meaningful about it.

Our results give a potential solution to this problem. Theorem~\ref{mainthm} shows that the generalized orbit of $x$ contains a weakly homogeneous point $x'$. For many questions about $\GL$-varieties, understanding one point in each generalized orbit is sufficient. Thus it should be enough to understand the fiber of $\phi$ above $x'$. This carries an action of the stabilizer of $x'$, which is linear oligomorphic. We expect that most of the known results about $\GL$-varieties should extend to $G$-varieties, where $G$ is a linear oligomorphic group (or at least the kind of linear oligomorphic group arising in this paper). Thus one should have numerous tools to study the fiber over $x'$.

\textit{(b) Module theory of $\GL$-algebras.} Let $A$ be a $\GL$-algebra, e.g., the coordinate ring of a $\GL$-variety. We would like to understand the category $\Mod_A$ of (equivariant) $A$-modules. For example, in the simplest case, $A=\Sym(\bV)$, the category $\Mod_A$ is equivalent to the category of $\mathbf{FI}$-modules \cite{fimodule}, and this category was studied in detail in \cite{symc1}. Some other small cases have been analyzed too, e.g., \cite{sym2noeth, symu1, symsp1}, and this all ties in with the representation theory of combinatorial categories \cite{catgb}.

Suppose that $A$ is a domain. One can then define the ``generic category'' $\Mod_A^{\rm gen}$ as the Serre quotient of $\Mod_A$ by the torsion subcategory. In all cases where $\Mod_A$ has been analyzed, the first step has been understanding the generic category; and generic categories have been studied by identifying them with representation categories of certain groups. For example, if $A=\Sym(\Sym^2(\bV))$ is the coordinate ring of $\bA^{[(2)]}$ then the generic category is identified with the category of algebraic representations of the infinite orthogonal group.

Prior to this paper, there was not a natural group associated to a general $A$. But our results furnish a group, namely, the automorphism group $G$ of the weakly homogeneous space associated to the generic generalized orbit in $\Spec(A)$. A natural conjecture is that $\Mod_A^{\rm gen}$ is equivalent to the category of algebraic representations of $G$. We proved this in \cite{homoten} (using results from \cite{tcares}) when $A$ is the coordinate ring of $\bA^{\ulambda}$.

\textit{(c) Tensor categories.} Deligne \cite{Deligne} constructed an interesting new tensor category $\uRep(\fS_t)$ by ``interpolating'' the representation theory of the finite symmetric groups. Motivated by Deligne's work, we gave a general construction of tensor categories based on oligomorphic groups \cite{repst}. Our theory recovers Deligne's category when applied to the infinite symmetric group. Other oligomorphic groups lead to fundamentally new tensor categories, such as the Delannoy category \cite{line, circle}.

Before Deligne's work, Deligne and Milne \cite{DeligneMilne} constructed a similar tensor category $\uRep(\GL_t)$ by interpolating the algebraic representation theory of general linear. This category does not fit into the theory of \cite{repst}. Our hope is that \cite{repst} admits a generalization to the linear oligomorphic case: the basic idea is that $\uRep(\GL_t)$ should come from the infinite general linear group, while other linear oligomorphic groups might lead to new tensor categories. For this reason, we are generally interested in constructing new linear oligomorphic groups. We find many new examples in this paper.

\subsection{Questions}

We mention a few open questions related to this paper:
\begin{enumerate}
\item The maximal isogeny class in a Zariski class contains a distinguished isomorphism class (the weakly homogeneous one). Do other isogeny classes contain a distinguished isomorphism class?
\item As mentioned, there is often a unique minimal isogeny class among universal $\ulambda$-spaces \cite{BDDE}. Does each Zariski class contain a unique minimal isogeny class?
\item Can the results of this paper be extended to more general coefficient fields (i.e., not algebraically closed, or of positive characteristic)?
\end{enumerate}

\subsection{Outline}

In \S \ref{s:gl}, we review material on $\GL$-varieties. In \S \ref{s:ten}, we begin the process of integrating the theories of tensor spaces and $\GL$-varieties. In \S \ref{s:age}, we study the ``age'' of a tensor space; this is another example of a concept we import from model theory. In \S \ref{s:homo}, we review generalities on homogeneous objects, and establish some basic properties of weakly homogeneous tensor spaces. The key technical result is proved in \S \ref{s:key}, and then the main theorems are deduced from it in \S \ref{s:main}.

\subsection{Notation}

We list some of the important notation:
\begin{description}[align=right,labelwidth=2cm,leftmargin=!]
\item [$k$] the base field, always algebraically closed of characteristic~0
\item [$\bV$] a $k$-vector space of dimension $\aleph_0$
\item [$\GL$] the group of all linear automorphisms of $\bV$
\item [$\ulambda$] a tuple of partitions
\item [$\bA^{\ulambda}$] the basic affine $\GL$-variety (\S \ref{ss:glvar})
\item [$\cC_{\ulambda}$] the category of $\ulambda$-spaces (\S \ref{ss:forms})
\end{description}

\section{$\GL$-varieties} \label{s:gl}

In \S \ref{s:gl}, we review the theory of $\GL$-varieties developed in \cite{polygeom}. We refer to \cite{polygeom2a, polygeom2b, charp} for additional results in this direction.

\subsection{Polynomial representations}

Let $\bV=\bigcup_{n \ge 1} k^n$ be a vector space of countable infinite dimension, and let $\GL=\GL(\bV)$ be the group of all automorphisms of $\bV$. For the purposes of \S \ref{s:gl}, one could instead use the ``small version'' of this group, namely $\bigcup_{n \ge 1} \GL_n$, but for our applications to tensor spaces it is essential to use the big version. A \defn{polynomial representation} of $\GL$ is a representation that occurs as a subquotient of a (possibly infinite) direct sum of tensor powers of $\bV$. We let $\Rep^{\pol}(\GL)$ denote the category of such representations. It is abelian and closed under tensor product.

For a partition $\lambda$, we let $\bS_{\lambda}$ denote the usual Schur functor. For a tuple $\ulambda=[\lambda_1, \ldots, \lambda_r]$ of partitions (often simply called a ``tuple''), we let $\bS_{\ulambda}=\bS_{\lambda_1} \oplus \cdots \oplus \bS_{\lambda_r}$. We call $\ulambda$ \defn{pure} if it contains no empty partitions. The irreducible polynomial representations are exactly the representations $\bS_{\lambda}(\bV)$, for arbitrary partitions $\lambda$. The category $\Rep^{\pol}(\GL)$ is semi-simple. It follows that every finite length polynomial representation is isomorphic to some $\bS_{\ulambda}(\bV)$.

Let $\Vec$ denote the category of all vector spaces and $\Fun(\Vec, \Vec)$ the category of all endofunctors of $\Vec$, which is an abelian category. For $n \in \bN$, let $T_n \colon \Vec \to \Vec$ be the functor given on objects by $T_n(V)=V^{\otimes n}$. A \defn{polynomial functor} is an object of $\Fun(\Vec, \Vec)$ that occurs as a subquotient of a (possibly infinite) direct sum of $T_n$'s. We let $\Pol$ denote the category of polynomial functors. The functor $\Pol \to \Rep^{\pol}(\GL)$ given by $F \mapsto F(\bV)$ is an equivalence of categories. In particular, if $M$ is a polynomial representation of $\GL$, we can evaluate the corresponding polynomial functor on an arbitrary vector space $V$; we denote this by $M\{V\}$.

\subsection{$\GL$-varieties} \label{ss:glvar}

A \emph{$\GL$-algebra} is a commutative algebra object in $\Rep^{\pol}(\GL)$. Thus, explicitly, it is a commutative (and associative and unital) $k$-algebra $R$ equipped with an action of $\GL$ by algebra homomorphisms such that it forms a polynomial representation of $\GL$. We say that a $\GL$-algebra $R$ is \emph{finitely $\GL$-generated} if it is generated as a $k$-algebra by the $\GL$-orbits of finitely many elements; equivalently, there is a surjection of $\GL$-algebras $\Sym(V) \to R$ for some finite length polynomial representation $V$.

An affine \defn{$\GL$-variety} is a reduced affine scheme $X$ over $k$ equipped with an action of the group $\GL$ such that $\Gamma(X, \cO_X)$ is a finitely $\GL$-generated $\GL$-algebra. A quasi-affine $\GL$-variety is a $\GL$-stable open subscheme of an affine $\GL$-variety. A \emph{morphism} of $\GL$-varieties is a $\GL$-equivariant map of schemes over $k$. Let $\ulambda=[\lambda_1, \ldots, \lambda_r]$ be a tuple. Then $R_{\ulambda}=\Sym(\bS_{\ulambda}(\bV))$ is a $\GL$-algebra that is reduced and finitely $\GL$-generated. We let $\bA^{\ulambda}$ be its spectrum, which is a $\GL$-variety. The $\GL$-varieties $\bA^{\ulambda}$ play the role of the basic affine spaces $\bA^n$ in classical algebraic geometry: every affine $\GL$-variety is isomorphic to a closed $\GL$-subvariety of some $\bA^{\ulambda}$.

The most important result currently known about $\GL$-varieties is Draisma's theorem \cite{Draisma}: a $\GL$-variety $X$ is \defn{topologically noetherian}, in the sense that the descending chain condition holds for $\GL$-stable closed subsets of $X$. One consequence of this theorem is that any closed subvariety of $X$ is defined (set-theoretically) by the vanishing of finitely many orbits of functions on $X$.

Let $X=\Spec(R)$ be an affine $\GL$-variety. For an arbitrary vector space $V$, we have an algebra $R\{V\}$, and we define $X\{V\}=\Spec(R\{V\})$. In this way, $X$ defines a contravariant functor from $\Vec$ to the category of schemes; of course, $X\{V\}$ is also functorial in $X$. If $V$ is finite dimensional then $X\{V\}$ is in fact reduced and of finite type over $k$, and thus an ordinary finite dimensional variety. This construction mostly works in the quasi-affine case too:

\begin{proposition} \label{prop:qaff-func}
The construction $X\{V\}$ is defined for any quasi-affine $\GL$-variety $X$ and vector space $V$. It is functorial for surjective maps of vector spaces, and arbitrary maps of quasi-affine $\GL$-varieties.
\end{proposition}

\begin{proof}
Suppose $X$ is quasi-affine. Write $X=X_0 \setminus X_1$, where $X_0$ is an affine $\GL$-variety, and $X_1 \subset X_0$ is a closed $\GL$-subvariety. This expression is canonical: $X_0$ is the spectrum of the coordinate ring of $X$, and $X_1=X_0 \setminus X$. We define $X\{V\}=X_0\{V\} \setminus X_1\{V\}$.

Suppose that $f \colon V \to W$ is a surjection of vector spaces, and choose a splitting $V=V' \oplus W$. Let $S$ (resp.\ $R$) be the coordinate ring of $X_0$ (resp.\ $X_1$), so that we have a surjection $S \to R$. The ring $S\{W\}$ is obtained from $S\{V\}$ by killing all elements of positive degree for the action of $\GL(V')$, and similarly for $R$. Thus $R\{W\}=S\{W\} \otimes_{S\{V\}} R\{V\}$, and so $X_1\{W\} = X_0\{W\} \times_{X_0\{V\}} X_1\{V\}$. Thus $f$ induces a map $X\{W\} \to X\{V\}$, as required.

Suppose now that $\phi \colon Y \to X$ is a map of quasi-affine $\GL$-varieties. Write $Y=Y_0 \setminus Y_1$ similar to the above. Then $\phi$ induces a map $\phi_0 \colon Y_0 \to X_0$ such that $Y_1$ contains $\phi_0^{-1}(X_1)=Y_0 \times_{X_0} X_1$. It follows that $Y_1\{V\}$ contains $Y_0\{V\} \times_{X_0\{V\}} X_1\{V\}$, since $(-)\{V\}$ is compatible with fiber products. Thus $\phi$ induces $Y\{V\} \to X\{V\}$, as required.
\end{proof}

We note that if $X$ is quasi-affine the $X\{V\}$ need not be functorial for injective linear maps. Indeed, we can have an injection $V \to W$ such that the induced map $\bA^{\ulambda}\{W\} \to \bA^{\ulambda}\{V\}$ maps non-zero points to the origin; thus $X=\bA^{\ulambda} \setminus \{0\}$ is not functorial for injections.

\subsection{Points and orbits} \label{ss:orb}

Let $X$ be an affine $\GL$-variety. In this paper, a \defn{point} of $X$ will always mean a $k$-point. For a point $x \in X$, we let $\ol{O}_x$ be the Zariski closure of its $\GL$-orbit; this is an irreducible closed $\GL$-subvariety of $X$ \cite[Proposition~3.1]{polygeom}. We say that two points $x,y \in X$ belong to the same \defn{generalized orbit} if $\ol{O}_x=\ol{O}_y$. This is an equivalence relation, and we let $O_x$ be the equivalence class (generalized orbit) of $x$. We note that $O_x$ is a dense subset of $\ol{O}_x$, and any $\GL$-stable open subset of $\ol{O}_x$ contains $O_x$ \cite[Proposition~3.4]{polygeom}.

A point $x$ of $X$ is \defn{$\GL$-generic} if it has dense orbit, meaning $\ol{O}_x=X$. If $\ulambda$ is a pure tuple then $\bA^{\ulambda}$ admits a $\GL$-generic point \cite[Proposition~3.13]{polygeom}; in fact, it admits many of them.

\subsection{The decomposition theorem}

One of the main results of \cite{polygeom} is the decomposition theorem, which shows that a $\GL$-variety, or, more generally, a morphism of $\GL$-varieties, can be decomposed into elementary pieces. We now recall the statement. For this, we let $G(m) \subset \GL$ be the subgroup consisting of matrices of the form
\begin{displaymath}
\begin{pmatrix} 1 & 0 \\ 0 & \ast \end{pmatrix}
\end{displaymath}
where the top left block has size $m \times m$. Note that $G(m)$ is isomorphic to $\GL$, so we can speak of $G(m)$-varieties.

An affine $\GL$-variety $X$ is \defn{elementary} if it is isomorphic to one of the form $B \times \bA^{\ulambda}$, where $B$ is an irreducible finite dimensional affine variety $B$ and $\ulambda$ is a pure tuple. A quasi-affine $\GL$-variety $X$ is \defn{locally elementary at $x \in X$} if there exists a $G(m)$-stable open affine neighborhood of $x$ that is elementary as a $G(m)$-variety for some $m$. We say that $X$ is \defn{locally elementary} if it is so at all points. The decomposition theorem for $\GL$-varieties \cite[Corollary~7.9]{polygeom} states that if $X$ is any quasi-affine $\GL$-variety then there exists a decomposition $X=\bigsqcup_{i=1}^n X_i$, where each $X_i$ is a locally closed $\GL$-subvariety of $X$ that is locally elementary.

Suppose now that $\phi \colon Y \to X$ is a morphism of $\GL$-varieties. We say that $\phi$ is \defn{elementary} if there are isomorphisms $Y=C \times \bA^{\umu}$ and $X=B \times \bA^{\ulambda}$, where $B$ and $C$ are irreducible affine varieties and $\ulambda \subset \umu$ are tuples, such that $\phi=\psi \times \pi$, where $\psi \colon C \to B$ is a surjective morphism of varieties and $\pi \colon \bA^{\umu} \to \bA^{\ulambda}$ is the projection map. We say that $\phi$ is \defn{locally elementary} at $y \in Y$ if there exists $m \ge 0$ and $G(m)$-stable open affine neighborhoods $V$ of $y$ and $U$ of $\phi(y)$ such that $\phi$ induces a $G(m)$-elementary map $V \to U$. We say that $\phi$ is \defn{locally elementary} if it is surjective and locally elementary at all $y \in Y$. The decomposition theorem for morphisms \cite[Theorem~7.8]{polygeom} states that there are decompositions $Y=\bigsqcup_{j=1}^m Y_j$ and $X=\bigsqcup_{i=1}^n X_i$, where each $Y_j$ and $X_i$ is locally closed and $\GL$-stable, such that for each $j$ there is some $i$ such that $\phi$ induces a locally elementary map $Y_j \to X_i$.

\subsection{Chevalley's theorem}

Let $X$ be a quasi-affine $\GL$-variety. A subset of $X$ is \defn{$\GL$-constructible} if it is a finite union of locally closed $\GL$-stable subsets. The following is a version of Chevalley's theorem for $\GL$-varieties (see \cite[Theorem~7.13]{polygeom}), and follows easily from the decomposition theorem.

\begin{theorem} \label{thm:chevalley}
Let $\phi \colon Y \to X$ be a morphism of quasi-affine $\GL$-varieties and let $C$ be a $\GL$-constructible subset of $Y$. Then $\phi(C)$ is a $\GL$-constructible subset of $X$.
\end{theorem}

\subsection{Typical morphisms} \label{ss:typical}

Let $X$ be an irreducible $\GL$-variety. A \defn{typical morphism} is a dominant morphism $\phi \colon B \times \bA^{\ulambda} \to X$, where $B$ is an irreducible variety and $\ulambda$ is a pure tuple, such that no proper closed $\GL$-subvariety of $B \times \bA^{\ulambda}$ maps dominantly to $X$. A typical morphism always exists \cite[Proposition~8.2]{polygeom}. Moreover, if $\psi \colon C \times \bA^{\umu} \to X$ is a second typical morphism then $\ulambda=\umu$ and $\dim(B)=\dim(C)$ \cite[Corollary~8.5]{polygeom}. We call $\ulambda$ the \defn{type} of $X$ and $\dim(B)$ the \defn{typical dimension} of $X$. We require the following additional information:

\begin{proposition} \label{prop:type-orb}
Let $X$ be an irreducible $\GL$-variety. The following are equivalent:
\begin{enumerate}
\item $X$ is an orbit closure, i.e., $X$ contains a $k$-point with dense orbit.
\item $X$ has typical dimension~0.	
\item There is a dominant map $\bA^{\umu} \to X$ for some pure tuple $\umu$.
\end{enumerate}
If these conditions hold then a typical morphism for $X$ is unique up to isomorphism.
\end{proposition}

\begin{proof}
(a) $\Rightarrow$ (b). Let $x \in X$ be a $k$-point with dense orbit, and let $\phi \colon B \times \bA^{\ulambda} \to X$ be a typical morphism. The image of $\phi$ contains a non-empty $\GL$-stable open set by Chevalley's theorem (Theorem~\ref{thm:chevalley}), and this set necessarily contains $x$ \cite[Proposition~3.4]{polygeom}. Let $(b, y)$ be a $k$-point of $B \times \bA^{\ulambda}$ mapping to $x$ (which exists by \cite[Proposition~7.15]{polygeom}). Then $\phi$ restricts to a dominant morphism $\{b\} \times \bA^{\ulambda}$, and so $B=\{b\}$ is zero-dimensional.

(b) $\Rightarrow$ (c). This follows from the definition of typical morphism.

(c) $\Rightarrow$ (a). Since $\umu$ is pure, there is a $k$-point of $\bA^{\umu}$ with dense orbit \cite[Proposition~3.13]{polygeom}, and its image will be such a point in $X$.

Finally, suppose these conditions hold, and let $\phi$ and $\psi$ be two typical morphisms $\bA^{\umu} \to X$. By \cite[Proposition~8.4]{polygeom}, there is a dominant map $\theta \colon \bA^{\umu} \to \bA^{\umu}$ such that $\phi \circ \theta = \psi$. By \cite[Corollary~6.5]{polygeom}, $\theta$ is an isomorphism.
\end{proof}

\section{Tensor spaces} \label{s:ten}

\subsection{Basic definitions} \label{ss:forms}

Let $V$ be a vector space. For a partition $\lambda$, a \defn{$\lambda$-structure} on $V$ is a linear map $\omega \colon \bS_{\lambda}(V) \to k$. For a tuple $\ulambda=[\lambda_1, \ldots, \lambda_r]$, a \defn{$\ulambda$-structure} on $V$ is a tuple $\omega=(\omega_1, \ldots, \omega_r)$ where $\omega_i$ is a $\lambda_i$-structure. A \defn{$\ulambda$-space} is a vector space equipped with a $\ulambda$-structure. We say that a $\ulambda$-space is \defn{infinite}, \defn{countable}, or \defn{finite} if its dimension is infinite, exactly $\aleph_0$, or finite. If $(V, \omega)$ and $(W, \eta)$ are two $\ulambda$-spaces, an \defn{embedding} $\alpha \colon V \to W$ is an injective linear map such that $\alpha^*(\eta_i)=\omega_i$ for each $1 \le i \le r$. In this way, we have a category $\cC_{\ulambda}$ of $\ulambda$-spaces. We let $\cC^{\rf}_{\ulambda}$ be the subcategory of finite dimensional $\ulambda$-spaces. 

We note that giving a $\ulambda$-structure on $\bV$ is exactly the same as giving a point of $\bA^{\ulambda}$; in other words, $\bA^{\ulambda}$ is the space of $\ulambda$-structures on $\bV$. This is the essential connection between tensor spaces and $\GL$-varieties.

\begin{remark} \label{rmk:pure-cat}
If $\lambda=\emptyset$ is the empty partition then a $\lambda$-structure on a vector space is a linear map from $k$ to $k$, which we'll think of as simply a number. If $(V, a)$ and $(W, b)$ are two $\lambda$-spaces then an embedding $V \to W$ exists if and only if $a=b$ and $\dim(V) \le \dim(W)$. Thus $\cC_{\lambda}$ is equivalent to $k \times \cC_{\emptyset}$, where $k$ is regarded as a discrete category and $\cC_{\emptyset}$ is the category of vector spaces and injective linear maps. On the other hand, if $\ulambda$ is a pure tuple then there is a unique $\ulambda$-structure on the zero space, and for any $\ulambda$-space $V$ there is a unique embedding $0 \to V$, i.e., 0 is the initial object of $\cC_{\ulambda}$. This explains, in part, why pure tuples are special.
\end{remark}

\subsection{The ideal of a tensor space} \label{ss:ideal}

Let $(V, \omega)$ be a $\ulambda$-space, and recall $R_{\ulambda}=\Sym(\bS_{\ulambda}(\bV))$. An element $f \in R_{\ulambda}\{k^n\}$ defines a function on $V^n$. Roughly speaking, this function will be a polynomial expression in the $\omega_i$'s evaluated on various subsets of the components of $V^n$. Formally, given $v_1, \ldots, v_n$, we define $f(v_1, \ldots, v_n)$ to be the image of $f$ under the ring homomorphism
\begin{displaymath}
R_{\ulambda}\{k^n\} \to R_{\ulambda}(V) \to k
\end{displaymath}
where the first map is induced by the linear map $k^n \to V$ defined by the $v_i$'s, and the second map is induced from the $\omega_i$'s. See the examples below to get a better idea of how this works. Taking a limit of this construction, we may regard $f \in R_{\ulambda}$ as a function on $V^{\infty}$. Note that $f$ belongs to $R_{\ulambda}\{k^n\}$ for some $n$, and the resulting function will depend on only the first $n$ components of $V^{\infty}$.

We are now able to introduce a concept that is fundamental to this paper:

\begin{definition}
For a $\ulambda$-space $V$, we define $\fI(V)$ to be the set of all functions $f \in R_{\ulambda}$ that vanish identically on $V^{\infty}$.
\end{definition}

Informally, one should think of $\fI(V)$ as the ideal of additional relations the $\ulambda$-structure on $V$ satisfies. One easily sees that $\fI(V)$ is a $\GL$-stable radical ideal of $R_{\ulambda}$; we will see below (Corollary~\ref{cor:I-prime}) that it is in fact a prime ideal.  We now look at some examples.

\begin{example}
Suppose $\ulambda=[(1),(1)]$. Then a $\ulambda$-space is a vector space $V$ equipped with two linear functions $\lambda_1, \lambda_2 \in V^*$. For an integer $i \ge 1$, we have functions on $V^{\infty}$ defined by $v_{\bullet} \mapsto \lambda_1(v_i)$ and $v_{\bullet} \mapsto \lambda_2(v_i)$. The functions defined by $R_{\ulambda}$ are exactly the polynomial expressions in these basic functions. The ideal $\fI(V)$ can thus detect all linear dependencies between $\lambda_1$ and $\lambda_2$, which is enough to determine $V$ up to isomorphism.
\end{example}

\begin{example}
Suppose $\ulambda=[(2)]$. Then a $\ulambda$-space is a vector space $V$ equipped with a symmetric bilinear form $\omega$. For integers $i, j \ge 1$, we have a function on $V^{\infty}$ defined by $v_{\bullet} \mapsto \omega(v_i, v_j)$. The functions defined by $R_{\ulambda}$ are exactly the polynomial expressions in these basic functions. Consider the function $f_n \in R_{\ulambda}$ given by
\begin{displaymath}
f_n(v_{\bullet})=\det(\omega(x_i,x_j))_{1 \le i,j \le n}.
\end{displaymath}
Then $f_n \in \fI(V)$ if and only if $\omega$ has rank $<n$. Thus $\fI(V)$ determines the rank of $V$, considered as an element of $\bN \cup \{\infty\}$. In fact, this is exactly the information contained in $\fI(V)$, if $V \ne 0$. This is not enough to determine the isomorphism type of $V$, as $\fI(V)$ does not detect the dimension of the radical of the form.
\end{example}

We now prove two simple properties of $\fI(V)$.

\begin{proposition} \label{prop:ideal-age}
Let $V$ be a $\ulambda$-space. Then $\fI(V) = \bigcap_{W \subset V} \fI(W)$, where the intersection is taken over all finite dimensional subspaces $W$ of $V$, with the induced $\ulambda$-structure.
\end{proposition}

\begin{proof}
An element of $R_{\ulambda}$ only depends on finitely many vectors, so it vanishes on $V$ if and only if it vanishes on all finite dimensional subspaces of $V$.
\end{proof}

\begin{proposition} \label{prop:count-ideal}
Let $V$ be a $\ulambda$-space. Then there exists a countable $\ulambda$-space $V_0$ such that $\fI(V)=\fI(V_0)$. If $V$ is infinite, we can take $V_0$ to be a subspace of $V$.
\end{proposition}

\begin{proof}
First suppose that $V$ is infinite. If $W$ is any subspace of $V$ then obviously $\fI(V) \subset \fI(W)$. Moreover, if $f \not\in \fI(V)$, then there is a tuple $(v_1, \ldots, v_n)$ witnessing this, i.e., $f(v_1, \ldots, v_n) \ne 0$, and if we let $W'=W+\operatorname{span}(v_1, \ldots, v_n)$ then $f \not\in \fI(W')$. Now, for any $n$ and $d$, the degree $d$ piece $R_{\ulambda}(k^n)_d$ of the ring $R_{\ulambda}(k^n)$ is finite dimensional. It follows from the above observation that we can find a finite dimensional subspace $W_{n,d}$ of $V$ such that
\begin{displaymath}
\fI(V) \cap R_{\ulambda}(k^n)_d = \fI(W_{n,d}) \cap R_{\ulambda}(k^n)_d.
\end{displaymath}
We can now take $V_0=\sum_{n,d} W_{n,d}$. (Note that $\fI$ is always a homogeneous ideal, since every polynomial representation is graded.)

Now suppose that $V$ is finite. Let $V_0$ be a countable vector space, choose a surjection $V_0 \to V$, and pull-back the $\ulambda$-structure on $V$ to $V_0$. One easily verifies $\fI(V_0)=\fI(V)$, as required.
\end{proof}

\subsection{The generalized orbit of a tensor space} \label{ss:tenorb}

Let $V$ be a $\ulambda$-space. We have just defined an ideal $\fI(V)$ of $R_{\ulambda}$ associated to $V$. Since $R_{\ulambda}$ is the coordinate ring of $\bA^{\ulambda}$, we can convert this ideal into a geometric object: we define $\ol{O}_V$ to be the vanishing locus of $\fI(V)$ inside of $\bA^{\ulambda}$. Thus, by definition, $\ol{O}_V$ is a closed $\GL$-subvariety of $\bA^{\ulambda}$.

\begin{proposition}
$\ol{O}_V$ is an orbit closure, i.e., of the form $\ol{O}_x$ for some point $x$ of $\bA^{\ulambda}$.
\end{proposition}

\begin{proof}
First suppose that $V$ is countable. We may as well assume $V=(\bV, \omega)$ for some $\ulambda$-structure $\omega$ on $\bV$. Let $f \in R_{\ulambda}$. We can regard $f$ as a function on $\bA^{\ulambda}$ and $\omega$ as a point of $\bA^{\ulambda}$, and evaluate $f$ at $\omega$ to get a number $f(\omega) \in k$. Alternatively, we can regard $f$ as a function on $V^{\infty}$ via the construction in \S \ref{ss:ideal}. These two constructs are related as follows:
\begin{displaymath}
f(\omega)=f(e_1, e_2, \ldots),
\end{displaymath}
where $e_i$ is the standard basis of $V$. Acting by $\GL$, we see that $f$ vanishes on the orbit of $\omega$ if and only $f(v_1, v_2, \ldots)=0$ whenever $v_1, v_2, \ldots$ is a basis of $\bV$. Since $f$ only depends on its first $n$ arguments (for some $n$), and the linear independent tuples in $V^n$ are Zariski dense (in every finite dimensional subspace), the latter condition is equivalent to $f$ vanishing identically on $V^{\infty}$. We thus see that $\fI(V)$ coincides with the ideal of $\ol{O}_{\omega}$, and so $\ol{O}_V=\ol{O}_{\omega}$, as required.

In general, there is a countable $\ulambda$-space $V_0$ with $\fI(V_0)=\fI(V)$ (Proposition~\ref{prop:count-ideal}), and so $\ol{O}_V=\ol{O}_{V_0}$ is an orbit closure.
\end{proof}

\begin{corollary} \label{cor:I-prime}
$\fI(V)$ is a prime ideal.
\end{corollary}

\begin{proof}
The ideal $\fI(V)$ is radical and $\ol{O}_V$ is irreducible \cite[Proposition~3.1]{polygeom}.
\end{proof}

We define $O_V$ to be the unique dense generalized orbit in $\ol{O}_V$; this makes sense due to the above proposition. We now introduce an important concept:

\begin{definition}
Two $\ulambda$-spaces $V$ and $W$ are \defn{Zariski equivalent} if $O_V=O_W$.
\end{definition}

Alternatively, $V$ and $W$ are Zariski equivalent if $\fI(V)=\fI(W)$. More generally, we say that $V$ is a \emph{Zariski specialization} of $W$ if $\fI(W) \subset \fI(V)$, or, equivalently, $\ol{O}_V \subset \ol{O}_W$. If $W \to V$ is an embedding then $\fI(V) \subset \fI(W)$, and so $W$ is a Zariski specialization of $V$. In particular, if $V$ and $W$ are isogenous then they are Zariski equivalent.

\subsection{Categories associated to subvarieties}

Let $C$ be a $\GL$-constructible set in $\bA^{\ulambda}$. We define $\cC_{\ulambda}(C)$ to be the full subcategory of $\cC_{\ulambda}$ spanned by $\lambda$-spaces $V$ for which $O_V \subset C$. Alternatively, for any vector space $V$, we have a subset $C\{V\}$ of $\bA^{\umu}\{V\}$, and a tensor space $(V, \omega)$ belongs $\cC_{\umu}(C)$ if and only if $\omega \in C\{V\}$.

We note one simple property of these categories here:

\begin{proposition} \label{prop:closed-cat}
We have the following:
\begin{enumerate}
\item If $C$ is a closed subset of $\bA^{\ulambda}$ then $\cC_{\ulambda}(C)$ is downwards closed, i.e., if $W \to V$ is an embedding of $\ulambda$-spaces and $V \in \cC_{\ulambda}(C)$ then $W \in \cC_{\ulambda}(C)$.
\item If $C$ is an open subset of $\bA^{\ulambda}$ then $\cC_{\ulambda}(C)$ is upwards closed, i.e., if $W \to V$ is an embedding of $\ulambda$-spaces and $W \in \cC_{\ulambda}(C)$ then $V \in \cC_{\ulambda}(C)$.
\end{enumerate}
\end{proposition}

\begin{proof}
(a) Let $\fa$ be the ideal of $C$. Then $V \in \cC_{\ulambda}(C)$ if and only if $\fa \subset \fI(V)$, and so the result follows. (b) Let $\fb$ be the ideal of $\bA^{\ulambda} \setminus C$. Then $V \in \cC_{\ulambda}(C)$ if and only if $\fb \not\subset \fI(V)$, and so the result follows.
\end{proof}

\subsection{Universality} \label{ss:univ}

We say that a $\ulambda$-space $V$ is \defn{universal} if every finite dimensional $\ulambda$-space embeds into it. By Remark~\ref{rmk:pure-cat}, a universal space can exist only if $\ulambda$ is pure; by \cite[Theorem~3.2]{homoten}, if $\ulambda$ is pure then a universal space does exist. We are interested in the following characterization of universal spaces:

\begin{theorem} \label{thm:univ}
A countable $\ulambda$-space $V$ is universal if and only if $O_V$ is Zariski dense.
\end{theorem}

\begin{proof}
This is \cite[Corollary~2.6.3]{BDDE}.
\end{proof}

The theorem shows that the concept of universality interacts very nicely with the algebraic geometry of tensor spaces. In particular, we have the following corollary:

\begin{corollary} \label{cor:univ-zar}
Suppose $\ulambda$ is pure. Then the universal $\ulambda$-spaces form a single Zariski class.
\end{corollary}

\begin{remark}
Let $(V, \omega)$ be a countable $\ulambda$-space. Then universality of $V$ is also equivalent to $\omega$ having infinite strength. See \cite[Definition~2.6.4]{BDDE} for the definition of strength and \cite[Corollary~2.6.5]{BDDE} for the equivalence of these conditions when $\ulambda$ is a single partition.
\end{remark}

\subsection{Algebraic functors}

Let $\phi \colon \bA^{\umu} \to \bA^{\ulambda}$ be a map of $\GL$-varieties. For each vector space $V$, there is an induced map $\bA^{\umu}\{V\} \to \bA^{\ulambda}\{V\}$, which allows us to convert a $\umu$-structure on $V$ into a $\ulambda$-structure on $V$. This construction yields a functor
\begin{displaymath}
\Phi \colon \cC_{\umu} \to \cC_{\ulambda}.
\end{displaymath}
We say that a functor is \defn{algebraic} if it arises in this manner.

\begin{proposition} \label{prop:alg-im}
Let $\phi$ and $\Phi$ be as above, and let $C=\im(\phi)$. Then $\Phi(\cC_{\umu})=\cC_{\ulambda}(C)$.
\end{proposition}

\begin{proof}
The containment $\Phi(\cC_{\umu}) \subset \cC_{\ulambda}(C)$ means $\phi(\bA^{\umu}\{V\}) \subset C\{V\}$ for all vector spaces $V$, which follows from Proposition~\ref{prop:qaff-func}. For the reverse, suppose that $(V,\omega)$ belongs to $\cC_{\ulambda}(C)$, so that $\omega \in C\{V\}$. If $V$ is countable, we may as well identify it with $\bV$, and then $\omega \in \im(\phi)$; we can thus find $\eta \in \bA^{\umu}$ such that $\eta=\phi(\omega)$, and then $(V,\omega)=\Phi(V, \eta)$. Now suppose $V$ is finite. We may as well identify $V$ with the span of the first $n$ basis vectors in $\bV$. Let $\tilde{\omega}$ be the pull-back of $\omega$ to $\bV$, let $\tilde{\eta} \in \bA^{\umu}$ satisfy $\phi(\tilde{\eta})=\tilde{\omega}$, and let $\eta$ be the restriction of $\tilde{\eta}$ to $V$. Then $\Phi(V,\eta)=(V,\omega)$. The uncountable case is similar, but we omit the details (we will not need it).
\end{proof}

\begin{example}
Let $\umu=[(1), (1)]$ and $\ulambda=[(2)]$. Consider the map $\phi \colon \bA^{\umu} \to \bA^{\ulambda}$ defined by $\phi(\ell_1, \ell_2)=\ell_1^2+\ell_2^2$, and let $\Phi \colon \cC_{\umu} \to \cC_{\ulambda}$ be the associated algebraic functor. This functor is described as follows. Let $(V, (\lambda_1, \lambda_2))$ be a $\umu$-space, that is, $V$ is a vector space and $\lambda_1, \lambda_2 \in V^*$ are linear forms on $V$. Then $\Phi(V)=(V, \omega)$ is the quadratic space with form $\omega(v)=\lambda_1(v)^2+\lambda_2(v)^2$. Thus $\Phi$ simply converts two linear forms into a quadratic form by summing their squares.
\end{example}

\section{The age of a tensor space} \label{s:age}

Motivated by ideas from model theory \cite[\S 2.1]{Macpherson}, we make the following definition:

\begin{definition}
The \defn{age} of a $\ulambda$-space $V$, denoted $\age(V)$, is the class of all finite $\ulambda$-spaces that embed into $V$.
\end{definition}

The following is our main results on ages. It uses the theory of typical morphisms (\S \ref{ss:typical}).

\begin{theorem} \label{thm:typical-age}
Let $V$ be an infinite $\ulambda$-space, let $\phi \colon \bA^{\umu} \to \ol{O}_V$ be a typical morphism, and let $C \subset \bA^{\ulambda}$ be the image of $\phi$. Then $\age(V)=\cC^{\rf}_{\ulambda}(C)$.
\end{theorem}

\begin{proof}
First suppose that $V$ is countable. Without loss of generality, $V=(\bV, \omega)$ where $\omega \in \bA^{\ulambda}$ has orbit closure $\ol{O}_V$. The image of $\phi$ contains a dense open subset of $\ol{O}_V$ by Chevalley's theorem (Theorem~\ref{thm:chevalley}), which necessarily contains $\omega$ (see \S \ref{ss:orb}). We can thus lift $\omega$ to a point $\eta \in \bA^{\umu}$; here we use \cite[Proposition~7.15]{polygeom} to ensure we can actually lift $\omega$ to a $k$-point. Since $\phi$ is typical, $\eta$ has dense orbit, and so $(\bV, \eta)$ is a universal $\umu$-space (Theorem~\ref{thm:univ}). Letting $\Phi$ be the algebraic functor associated to $\phi$, it follows that $\age(V)$ contains $\Phi(\cC^{\rf}_{\umu})$, which is identified with $\cC^{\rf}_{\ulambda}(C)$ by (the proof of) Proposition~\ref{prop:alg-im}. On the other hand, if $W$ is a finite dimensional subspace of $V$ then $(W, \omega)=\Phi(W, \eta)$, and so $\age(V)$ is contained in $\Phi(\cC^{\rf}_{\umu})$.

We now handle the general case. Let $V_0$ be a countable subspace of $V$ with $\fI(V)=\fI(V_0)$ (Proposition~\ref{prop:count-ideal}). Of course, $\age(V_0) \subset \age(V)$. Suppose $W \in \age(V)$, and identify $W$ with a subspace of $V$. Then $\fI(V_0+W)=\fI(V_0)$, and so $\age(V_0+W)=\age(V_0)$ by the countable case, and so $W \in \age(V_0)$. Thus $\age(V_0)=\age(V)$. By the previous paragraph, $\age(V_0)=\cC^{\rf}_{\ulambda}(C)$, and so the same holds for $\age(V)$. (Note that $\ol{O}_V=\ol{O}_{V_0}$, and so they have the same typical morphism $\phi$.)
\end{proof}

The theorem shows that $\age(V)$ is a very well-behaved object; for instance, the identification with $\cC_{\ulambda}^{\rf}(C)$ shows that $\age(V)$ admits a finite algebraic description. The constructible set $C$ is unique (since typical morphisms are unique up to isomorphism by Proposition~\ref{prop:type-orb}), and it will be useful to give it a name:

\begin{definition}
We define $\uage(V) \subset \bA^{\ulambda}$ to be the $\GL$-constructible set $C$ above.
\end{definition}

The theorem has the following interesting corollary, which we essentially used in proof of the uncountable case.

\begin{corollary} \label{prop:Zar-age}
Let $V$ and $W$ be infinite $\ulambda$-spaces. Then $V$ and $W$ are Zariski equivalent if and only if $\age(V)=\age(W)$.
\end{corollary}

\begin{proof}
If $\age(V)=\age(W)$ then $\fI(V)=\fI(W)$ by Proposition~\ref{prop:ideal-age}. Conversely, if $\fI(V)=\fI(W)$ then $\age(V)=\age(W)$ by Theorem~\ref{thm:typical-age}; indeed, the description of $\age(V)$ given there only depends on $\fI(V)$.
\end{proof}

\begin{remark}
In light of the corollary, one might wonder if the following two statements are equivalent:
\begin{enumerate}
\item $V$ is a Zariski specialization of $W$.
\item $\age(V) \subset \age(W)$.
\end{enumerate}
Certainly (b) implies (a) by Proposition~\ref{prop:ideal-age}. However, the converse need not hold. To see this, let $\phi \colon \bA^{\umu} \to \bA^{\ulambda}$ be a map of $\GL$-varieties, with $\umu$ and $\ulambda$ pure, such that $\im(\phi)$ is not Zariski closed. For instance, if $\ulambda=[(4)]$ then the strength $\le 3$ locus in $\bA^{\ulambda}$ is not Zariski closed \cite{BBOV}, and is naturally the image of a map $\phi \colon \bA^{\umu} \to \bA^{\ulambda}$ with $\umu$ pure. Let $\omega \in \bA^{\ulambda}$ be the image of a $\GL$-generic point of $\bA^{\umu}$ and let $\eta$ be a point of $\bA^{\ulambda}\{k^n\}$, for some $n$, in $\ol{\im(\phi)} \setminus \im(\phi)$. Then $(\bV, \eta)$ is a Zariski specialization of $(\bV, \omega)$, but the age of the former contains $(k^n, \eta)$, and the age of the latter does not.
\end{remark}

\section{Homogeneous structures} \label{s:homo}

\subsection{Generalities}

Let $\cC$ be a category in which all morphisms are monomorphisms; we refer to morphisms in $\cC$ as embeddings. Let $\Ind(\cC)$ denote the category of ind-objects in $\cC$, and identify $\cC$ with a full subcategory of $\Ind(\cC)$. We say that an ind-object is \defn{countable} if it is represented by a inductive system in $\cC$ of countable cardinality; such an object is isomorphic to one of the form $X_1 \to X_2 \to \cdots$, with $X_i \in \cC$. We will mostly be concerned with the countable case.

Let $\Omega$ be an ind-object in $\cC$. We make the following definitions:
\begin{itemize}
\item $\Omega$ is \defn{universal} if every object of $\cC$ embeds into $\cC$.
\item $\Omega$ is \defn{f-injective} if whenever we have embeddings $\alpha \colon X \to Y$ and $\gamma \colon X \to \Omega$, with $X$ and $Y$ in $\cC$, there exists an embedding $\beta \colon Y \to \Omega$ such that $\gamma=\beta \circ \alpha$.
\item $\Omega$ is \defn{homogeneous} if whenever $\alpha, \beta \colon X \to \Omega$ are embeddings, with $X$ in $\cC$, there exists an automorphism $\sigma$ of $\Omega$ such that $\beta=\sigma \circ \alpha$.
\end{itemize}
We have the following basic results concerning these concepts (see \cite[\S A.4]{homoten}):

\begin{proposition} \label{prop:homo}
We have the following:
\begin{enumerate}
\item Any f-injective countable ind-object is homogeneous.
\item A universal homogeneous countable ind-object is f-injective.
\item Any two universal homogeneous countable ind-objects are isomorphic.
\item If $\Omega$ is a universal homogeneous countable ind-object then any countable ind-object embeds into $\Omega$.
\end{enumerate}
\end{proposition}

Categorical Fra\"iss\'e theory, which goes back to work of Droste--G\"obel \cite{DrosteGobel}, gives a criterion for the existence of a universal homogeneous countable ind-object (see \cite[\S A.6]{homoten}). This was used in \cite{homoten} to construct universal homogeneous $\umu$-spaces (see below), but will not be needed in this paper; we construct our weakly homogeneous $\ulambda$-spaces by modifying the universal homogeneous $\umu$-spaces constructed in \cite{homoten}.

\subsection{Homogeneous tensor spaces}

Fix a tuple of partitions $\umu$. We now apply the above concepts when $\cC=\cC_{\umu}^{\rf}$, identifying $\Ind(\cC)$ with $\cC_{\umu}$. In particular, we have the following important concept:

\begin{definition} \label{defn:homo}
A $\umu$-space $V$ is \emph{homogeneous} if whenever $\alpha,\beta \colon W \to V$ are embeddings, with $W$ a finite $\umu$-space, there exists an automorphism $\sigma$ of $V$ such that $\beta=\sigma \circ \alpha$.
\end{definition}

\begin{theorem} \label{thm:univ-homo}
Suppose $\umu$ is pure. Then a countable universal homogeneous $\umu$-space exists, and any two are isomorphic.
\end{theorem}

\begin{proof}
This is the main theorem of \cite{homoten}. The proof uses categorical Fra\"iss\'e theory.
\end{proof}

\begin{example} \label{ex:pow-homo}
Let $\umu=[(d)]$ and let $\ulambda=[(de)]$, where $d,e \ge 1$ are integers. Let $(V,g)$ be a countable homogeneous $\umu$-space, and put $f=g^e$. We claim that $(V, f)$ is a homogeneous $\ulambda$-space. To see this, suppose that $W$ and $W'$ are two finite dimensional subspaces of $V$, and $\alpha \colon W \to W'$ is an isomorphism of $\ulambda$-spaces. We must construct an automorphism $\sigma$ of $(V,f)$ extending $\alpha$.

By assumption, we have $g(\alpha x)^e=g(x)^e$ for all $x \in W$. Since this an equality of polynomial functions, there is an $e$th root of unity $\epsilon$ such that $g(\alpha x) = \epsilon g(x)$ holds for all $x \in W$. Let $\delta$ be a $d$th root of $\epsilon$, and let $\beta \colon W \to W'$ be defined by $\beta(x)=\delta^{-1} \alpha(x)$. Since $g$ is homogeneous of degree $d$, we have $g(\beta x)=g(x)$ for all $x \in W$. Thus $\beta$ is an isomorphism of $\umu$-spaces. Since $(V,g)$ is homogeneous, there is an automorphism $\tau$ of $(V,g)$ extending $\beta$. We can now define $\sigma$ by $\sigma(x) = \delta \tau(x)$. Note that $f(\delta x)=f(x)$ since $f$ is homogeneous of degree $de$ and $\delta$ is a $de$th root of unity.
\end{example}

\begin{remark}
The above example this contradicts \cite[Proposition~3.13]{homoten}. There was an error in that proposition: in the proof there, we chose the constants $c_i$ to have $c_1 \ne 0$, but this is only guaranteed to be possible if $r>1$. The argument there still proves the following weaker statement: if $(V,f)$ is a homogeneous $[(d)]$-space then either $(V,f)$ is universal, or $f$ is a power of some lower degree form.
\end{remark}

\begin{example} \label{ex:quad-homo}
Suppose $\umu=[(2)]$. Let $V$ be a countable quadratic space. Up to isomorphism, $V$ is determined by two invariants: its rank $r$, and the dimension of its nullspace $d$. Of course, at least one of $r$ or $d$ must be infinite. We examine the various cases:
\begin{itemize}
\item If $r=\infty$ and $d=0$ then $V$ is universal homogenous. This follows easily from Witt's theorem; see \cite[Example~3.10]{homoten}.
\item If $r \ge 2$ and $d>0$ then $V$ is not homogeneous. Indeed, in this case we can find isotropic vectors $v_1$ and $v_2$ such that $v_1$ is not in the nullspace but $v_2$ is. The one-dimensional subspaces spanned by $v_1$ and $v_2$ are isomorphic, but there is no automorphism of $V$ that moves $v_1$ to $v_2$.
\item If $r=1$ then $V$ is homogeneous by Example~\ref{ex:pow-homo}. In this case any isotropic vector belongs to the nullspace, so the mechanism in the previous case is not available.
\item Finally, if $r=0$ then $V$ is obviously homogeneous: the form is identically zero. \qedhere
\end{itemize}
\end{example}

\begin{example}
Suppose $\umu=[(2),(2)]$. Let $(A,B)$ be a pair of infinite symmetric matrices, regarded as a $\umu$-structure on $\bV$. One can show that $(A,B)$ gives the universal homogeneous structure if and only if the columns of $A$ and $B$ (taken all together) are linearly independent. See \cite[Example~3.11]{homoten}.
\end{example}

\subsection{Weakly homogeneous spaces}

We now introduce a variant of Definition~\ref{defn:homo}.

\begin{definition} \label{defn:weak-homo}
A $\ulambda$-space $V$ is \defn{weakly homogeneous} if there exists a dense $\GL$-stable open subset $U$ of $\uage(V)$ such that the homogeneity condition holds in $\cC_{\ulambda}(U)$; that is, if $\alpha,\beta \colon W \to V$ are two embeddings with $W$ in $\cC^{\rf}_{\ulambda}(U)$ then there exists $\sigma \in \Aut(V)$ such that $\beta=\sigma \circ \alpha$.
\end{definition}

Let $V$ be weakly homogeneous $\Sigma$ be the class of $U$'s that fulfill the condition in the above definition. Clearly, $\Sigma$ is closed under arbitrary unions, and thus contains a unique maximal element. We denote this maximal element by $\cU(V)$.

\begin{example} \label{ex:wh-fin-rk}
Suppose $\ulambda=[(2)]$, and let $V$ be a countable quadratic space of rank $r$ and nullspace dimension $d$. There are a few cases to analyze:
\begin{itemize}
\item If $r$ is finite then $V$ is weakly homogenous. Indeed, one can show that if $W$ is a finite dimensional quadratic space of rank exactly $r$ then for any two embeddings $\alpha, \beta \colon W \to V$ there is an automorphism $\sigma$ of $V$ such that $\beta=\sigma \circ \alpha$. The space $\uage(V)$ is the rank $\le r$ locus in $\bA^{\ulambda}$, and one can take the $U$ in Definition~\ref{defn:weak-homo} to be the rank $r$ locus in $\bA^{\ulambda}$.
\item We have already seen that if $r=\infty$ and $d=0$ then $V$ is homogenous.
\item Finally, if $r=\infty$ and $d>0$ then $V$ is not weakly homogeneous. Indeed, in this case $\uage(V)=\bA^{\ulambda}$. Suppose $U$ is a non-empty open $\GL$-stable subset of $\bA^{\ulambda}$. Then $U$ is the rank $\ge s$ locus for some $s$. Choose a finite dimensional quadratic space of rank at least $s$ with non-trivial nullspace. We can then find embeddings $\alpha, \beta \colon W \to V$ such that the image of $\alpha$ does not meet the nullspace of $V$, but the image of $\beta$ does. Hence there is no automorphism $\sigma$ such that $\beta=\sigma \circ \alpha$. \qedhere
\end{itemize}
\end{example}

\begin{proposition} \label{prop:weak-homo-age}
Let $V$ and $W$ be countable weakly homogeneous $\ulambda$-spaces. The following are equivalent:
\begin{enumerate}
\item $V$ and $W$ are isomorphic.
\item $V$ and $W$ are Zariski equivalent.
\item $\age(V)=\age(W)$.
\end{enumerate}
\end{proposition}

\begin{proof}
It is clear that (a) implies (c), and we have already seen that (b) and (c) are equivalent (Proposition~\ref{prop:Zar-age}). We now assume (b) and (c) hold and prove (a). Let $U=\cU(V) \cap \cU(W)$, which is open dense subset of $\uage(V)=\uage(W)$. Both $V$ and $W$ are universal homogeneous objects in the category $\cC_{\ulambda}(U)$, and therefore isomorphic (Proposition~\ref{prop:homo}(c)).
\end{proof}

\begin{proposition}
Let $V$ be a countable weakly homogeneous $\ulambda$-space. Then any countable $\ulambda$-space $W$ with $O_W \subset \cU(V)$ embeds into $V$.
\end{proposition}

\begin{proof}
Since $\age(W)$ is dense in $\ol{O}_W$, it must meet the open set $\cU(V)$. There is thus a finite dimensional subspace $W_1 \subset W$ that belongs to $\cC_{\ulambda}(\cU(V))$. Any other subspace of $W$ containing $W_1$ also belongs to $\cC_{\ulambda}(\cU(V))$ (Proposition~\ref{prop:closed-cat}(b)), and so $W$ is a countable ind-object of $\cC_{\ulambda}(\cU(V))$. It therefore embeds into $V$ by Proposition~\ref{prop:homo}(d).
\end{proof}

\begin{corollary} \label{cor:wh-embed}
Let $V$ be a countable weakly homogeneous $\ulambda$-space and let $W$ be a countable $\ulambda$-space that is Zariski equivalent to $V$. Then $W$ embeds into $V$.
\end{corollary}

\begin{proof}
We have $O_W=O_V$ and $O_V \subset \cU(V)$, so the result follows from the proposition.
\end{proof}

\section{The key result} \label{s:key}

The following is the key result that will allow us to construct weakly homogeneous spaces.

\begin{theorem} \label{thm:key}
An algebraic functor $\Phi \colon \cC_{\umu} \to \cC_{\ulambda}$, with $\umu$ pure, maps the countable universal homogeneous $\umu$-space to a weakly homogeneous $\ulambda$-space.
\end{theorem}

The proof will take the entirety of \S \ref{s:key}. The basic idea is as follows. Let $\phi \colon \bA^{\umu} \to \bA^{\ulambda}$ be the map associated to $\Phi$, and let $E$ be a universal homogeneous $\umu$-space. Using the decomposition theorem (applied to $\phi$), we show that generic members of the age of $\Phi(E)$ can be lifted back to $\cC_{\umu}$ in a sufficiently nice way (see Lemma~\ref{lem:key-3} for the precise statement), and this allows us to transfer the f-injective property of $E$ to $\Phi(E)$.

\begin{lemma} \label{lem:key-1}
Let $\phi \colon Y \to X$ be an elementary morphism of $\GL$-varieties. Let $W \subset V$ be vector spaces, and consider the diagram
\begin{displaymath}
\xymatrix{
Y\{V\} \ar[r] \ar[d]_{\phi} & Y\{W\} \ar[d]^{\phi} \\
X\{V\} \ar[r] & X\{W\} }
\end{displaymath}
Let $\ol{y} \in Y\{W\}$ and $x \in X\{V\}$ be such that $\phi(\ol{y}) = x \vert_W$. Then there exists $y \in Y\{V\}$ such that $y \vert_W=\ol{y}$ and $\phi(y)=x$.
\end{lemma}

\begin{proof}
By definition, we can choose isomorphisms
\begin{displaymath}
Y=C \times \bA^{\ulambda} \times \bA^{\umu}, \qquad X=B \times \bA^{\ulambda}
\end{displaymath}
such that $\phi=\psi \times \pi$, where $\psi \colon C \to B$ is a surjection of varieties and $\pi$ is the projection map. Write $\ol{y}=(c, w, w')$ where $c \in C$, $w \in \bA^{\ulambda}\{W\}$, and $w' \in \bA^{\umu}\{W\}$, so that $\phi(\ol{y})=(\psi(c), w)$. Write $x=(b, v)$, where $b \in B$ and $v \in \bA^{\ulambda}\{V\}$. We have $x \vert_W = (b, v \vert_W)$, so $b=\psi(c)$ and $v \vert_W=w$. We take $y=(c, v, v')$, where $v' \in \bA^{\umu}\{V\}$ is any element that restricts to $w'$.
\end{proof}

For a quasi-affine scheme $X$, we let $\ol{X}$ be the spectrum of $\Gamma(X, \cO_X)$. This is an affine scheme containing $X$ as an open subscheme. If $X$ is a quasi-affine $\GL$-variety then $\ol{X}$ is an affine $\GL$-variety.

\begin{lemma} \label{lem:key-4}
Let $\phi \colon Y \to X$ be a locally elementary morphism of $\GL$-varieties, let $y \in Y$, and let $x = \phi(y)$. Then there exists $m \ge 0$ and $G(m)$-stable open affine neighborhoods $Y_0$ of $y$ and $X_0$ of $x$ such that $\phi$ induces a $G(m)$-elementary map $Y_0 \to X_0$. Moreover, we can take $X_0=\ol{X}[1/h]$ for some $G(m)$-invariant function $h$ on $\ol{X}$.
\end{lemma}

\begin{proof}
By definition, there exists $m \ge 0$ and open affine $G(m)$-stable neighborhoods $V$ of $y$ and $U$ of $x$ such that $\phi$ induces an elementary map $V \to U$. Now, $U$ is a union of basic opens $\ol{X}[1/h_i]$ for $i$ in some index set $I$. Since $U$ is quasi-compact, we can take $I$ to be finite. We can thus increase $m$ so that each $h_i$ is $G(m)$-invariant; note that $\phi$ remains elementary by \cite[Proposition~7.3]{polygeom}. Since $U$ is $G(m)$-elementary and $h_i$ is a $G(m)$-invariant function on $U$, it follows that $U[1/h_i]=\ol{X}[1/h_i]$ is also $G(m)$-elementary. Let $h=h_i$ be such that $x$ belongs to $X_0=\ol{X}[1/h]$, and let $Y_0=V \cap \phi^{-1}(X_0)$. Then $\phi \colon Y_0 \to X_0$ is elementary by Step~1 in the proof of \cite[Proposition~7.5]{polygeom}, and clearly $y \in Y_0$.
\end{proof}

Suppose $U$ is a finite dimensional vector space. Given a polynomial functor $F$, we define the \defn{shift} $\Sh_U{F}$, to be the functor $V \mapsto F\{U \oplus V\}$. This is again a polynomial functor. This construction induces a shift operation on polynomial representations, $\GL$-algebras, and affine $\GL$-varieties; see \cite[\S 4.1]{polygeom} for details. In particular, if $X$ is an affine $\GL$-variety then the pair $(\GL, \Sh_{k^m}{X})$ is isomorphic to the pair $(G(m), X)$, that is, restricting from $\GL$ to $G(m)$ is essentially the same as shifting by $k^m$. This point of view will be useful in the next proof.

\begin{lemma} \label{lem:key-2}
Let $\phi \colon Y \to X$ be a locally elementary morphism of (quasi-affine) $\GL$-varieties. Let $W \subset V$ be vector spaces, and consider the diagram
\begin{displaymath}
\xymatrix{
Y\{V\} \ar@{..>}[r] \ar[d]_{\phi} & Y\{W\} \ar[d]^{\phi} \\
X\{V\} \ar@{..>}[r] & X\{W\} }
\end{displaymath}
Let $\ol{y} \in Y\{W\}$ and $x \in X\{V\}$ be such that $\phi(\ol{y}) = x \vert_W$. Then there exists $y \in Y\{V\}$ such that $y \vert_W=\ol{y}$ and $\phi(y)=x$.
\end{lemma}

\begin{proof}
The affine $\GL$-variety $\ol{X}$ is functorial for maps of vector spaces, and so there is an induced map $\ol{X}\{V\} \to \ol{X}\{W\}$. However, this map need not carry $X\{V\}$ into $X\{W\}$, which is indicated by the dotted arrow. However, our point $x \in X\{V\}$ does map into $X\{W\}$.

We first suppose that $W$ is infinite dimensional. By the Lemma~\ref{lem:key-4}, we can find a decomposition $W=U \oplus W_1$ with $U$ finite dimensional, and open affine $\GL$-subvarieties $Y_0$ of $\Sh_U(\ol{Y})$ and $X_0$ of $\Sh_U(\ol{X})$ such that:
\begin{itemize}
\item $X_0=(\Sh_U{\ol{X}})[1/h]$ for some function $h$ on $\ol{X}\{U\}$.
\item $\phi$ induces an elementary map $Y_0 \to X_0$ of $\GL$-varieties.
\item $y$ belongs to $Y_0\{W_1\}$.
\end{itemize}
Choose a decomposition $V=V' \oplus W$, and put $V_1=V' \oplus W_1$. Consider the diagram
\begin{displaymath}
\xymatrix{
Y_0\{V_1\} \ar[r] \ar[d]_{\phi} & Y_0\{W_1\} \ar[d]^{\phi} \\
X_0\{V_1\} \ar[r] & X_0\{W\} }
\end{displaymath}
Now, the function $h$ on $(\Sh_U{\ol{X}})\{V_1\}$ is pulled back from $(\Sh_U{\ol{X}})\{W_1\}$. It follows that $h$ is non-vanishing on $x \in (\Sh_U{\ol{X}})\{V_1\}$, since it is non-vanishing on the image $\ol{x}$ of $x$; we thus find that $x$ belongs to $X_0\{V_1\}$. The result now follows from Lemma~\ref{lem:key-1}.

We now treat the case where $W$ is finite dimensional. Let $U$ be an infinite dimensional vector space, and put $W'=W \oplus U$ and $V'=V \oplus U$. Let $x'$ be the image of $x$ under the map $X\{V\} \to X\{V'\}$ induced by the quotient $V' \to V$; note that $X$ is functorial for such maps. Similarly define $\ol{y}' \in Y\{W'\}$. By the above case, we obtain $y' \in Y\{V'\}$ such that $\phi(y')=x'$ and $y' \vert_{W'} = \ol{y}'$. Working in $\ol{Y}$, we have
\begin{displaymath}
(y' \vert_V) \vert_W = (y' \vert_{W'}) \vert_W = \ol{y}' \vert_W = \ol{y},
\end{displaymath}
which belongs to $Y\{W\}$. It follows that $y=y' \vert_V$ belongs to $Y\{V\}$ (if it belonged to the complement, then so would its restriction to $W$, which is $\ol{y}$, and this is not the case), and $y \vert_W=\ol{y}$. We also have
\begin{displaymath}
\phi(y) = \phi(y' \vert_V) = \phi(y') \vert_V = x' \vert_V = x.
\end{displaymath}
We thus see that $y$ is the sought after element of $Y\{V\}$.
\end{proof}

\begin{lemma} \label{lem:key-3}
Let $\phi \colon Y \to X$ be a locally elementary morphism of quasi-affine $\GL$-varieties, and let $Y \subset \bA^{\umu}$ and $X \subset \bA^{\ulambda}$ be locally closed embeddings. Suppose $(W, \eta_W)$ is an object of $\cC_{\umu}(Y)$ and $(V, \omega_V)$ is an object of $\cC_{\ulambda}(X)$, and we have an embedding $i \colon (W, \omega_W) \to (V, \omega_V)$ of $\ulambda$-spaces, where $\omega_W=\phi(\eta_W)$. Then there exists a $\umu$-structure $\eta_V$ on $V$ such that $(V,\eta_V) \in \cC_{\umu}(Y)$, $\phi(\eta_V)=\omega_V$, and $i \colon (W,\eta_W) \to (V,\eta_V)$ is a map of $\umu$-spaces.
\end{lemma}

\begin{proof}
This is simply a reformulation of Lemma~\ref{lem:key-2}
\end{proof}

\begin{proof}[Proof of Theorem~\ref{thm:key}]
Let $(E,\eta_E)$ be a countable universal homogeneous $\umu$-space. Put $\omega_E=\phi(\eta_E)$. We show that $(E, \omega_E)$ is a weakly homogeneous $\ulambda$-space. In what follows, $\ulambda$-structures are denoted with $\omega$ and $\umu$-structures with $\eta$.

We have a dominant map of $\GL$-varieties $\phi \colon \bA^{\umu} \to \ol{O}(E,\omega_E)$. Applying the decomposition theorem, we find that there are non-empty open $\GL$-subvarieties $Y \subset \bA^{\umu}$ and $X \subset \ol{O}(E,\omega_E)$ such that $\phi$ induces a locally elementary map $Y \to X$. Note that $\uage(E, \omega_E)$ is the image of $\phi$, and $X$ is an open subset of this. We claim that $X$ witnesses the weak homogeneity of $(E, \omega_E)$.

Let $\alpha \colon W \to E$ and $\gamma \colon W \to V$ be embeddings of $\ulambda$-spaces, with $(W, \omega_W)$ and $(V, \omega_V)$ in $\cC^{\rf}_{\ulambda}(X)$. Define $\eta_W=i^*(\eta_E)$, so that $\alpha$ is also an embedding of $\umu$-spaces. Let $\eta_V$ be the $\umu$-structure on $V$ produced by Lemma~\ref{lem:key-3}. Since $(E, \eta_E)$ is f-injective (Proposition~\ref{prop:homo}(b)), there is an embedding $\beta \colon V \to E$ of $\umu$-spaces such that $\gamma=\beta \circ \alpha$. Applying $\Phi$, we obtain a similar factorization in $\cC_{\ulambda}$. This shows that $(E, \omega_E)$ is an f-injective object of $\cC_{\ulambda}(X)$. It is therefore a homogeneous object of this category (Proposition~\ref{prop:homo}(a)), and therefore a weakly homogeneous $\ulambda$-space.
\end{proof}

\section{The main theorems} \label{s:main}

Our main results now follow with little difficulty.

\begin{theorem} \label{thm:main1}
Let $\cZ$ be a Zariski class of $\ulambda$-spaces. Then:
\begin{enumerate}
\item The class $\cZ$ contains a countable weakly homogeneous $\ulambda$-space $V$
\item Any other countable weakly homogeneous space in $\cZ$ is isomorphic to $V$.
\item There exists an algebraic functor $\Phi \colon \cC_{\umu} \to \cC_{\ulambda}$, for some pure tuple $\umu$, such that $V \cong \Phi(W)$, where $W$ is the countable universal homogeneous $\umu$-space.
\end{enumerate}
\end{theorem}

\begin{proof}
Let $V_0$ belong to $\cZ$, let $\phi \colon \bA^{\umu} \to \ol{O}(V_0)$ be a typical morphism, and let $\Phi \colon \cC_{\umu} \to \cC_{\ulambda}$ be the associated algebraic functor. Let $W$ be a universal homogeneous $\umu$-space. Then $\Phi(W)$ is a countable weakly homogeneous space (Theorem~\ref{thm:key}), and of class $\cZ$ since $\phi$ maps dominantly to $\ol{O}(V)$. Note that if $x \in \bA^{\umu}$ corresponds to $W$ then $\phi(x)$ belongs to $O(V)$, and corresponds to $\Phi(W)$. This proves (a) and (c), and we have already seen (b) in Proposition~\ref{prop:weak-homo-age}.
\end{proof}

The class of $\ulambda$-orders admits a quasi-order $\le$ by defining $V \le W$ if there exists an embedding $V \to W$. The equivalence relation induced by this quasi-order is called \defn{isogeny}, and $\le$ induces a partial order on the collection of isogeny classes. It is clear that Zariski equivalent spaces are isogenous (\S \ref{ss:tenorb}), and so isogeny is finer than Zariski equivalence.

\begin{corollary}
In each Zariski class $\cZ$ of $\ulambda$-spaces, there is a unique maximal isogeny class of countable $\ulambda$-space, namely, that of the weakly homogeneous countable $\ulambda$-space.
\end{corollary}

\begin{proof}
Let $V$ be the weakly homogeneous countable $\ulambda$-space in $\cZ$. If $W$ is any countable $\ulambda$-space in $\cZ$ then $W$ embeds into $V$ (Corollary~\ref{cor:wh-embed}), and so the result follows.
\end{proof}

Using the above theorem, we can improve Theorem~\ref{thm:key}.

\begin{corollary}
An algebraic functor $\Phi \colon \cC_{\umu} \to \cC_{\ulambda}$ maps countable weakly homogeneous spaces to weakly homogeneous spaces.
\end{corollary}

\begin{proof}
Let $V$ be a countable weakly homogeneous $\umu$-space. Let $\Psi \colon \cC_{\unu} \to \cC_{\umu}$ be an algebraic functor such that $V \cong \Psi(W)$ for some universal homogeneous $\unu$-space $W$. Then $\Phi(V) \cong \Psi(\Phi(W))$ is weakly homogeneous by Theorem~\ref{thm:key}.
\end{proof}

\begin{remark}
The theory of typical morphisms shows that the $\umu$ in Theorem~\ref{thm:main1}(c) is unique, if it is taken to be minimal. We can therefore define the \defn{type} of a weakly homogeneous countable $\ulambda$-space $V$ to be the pure tuple $\umu$ from this statement.
\end{remark}

We now turn to automorphism groups. Recall the following definition from \cite{homoten}:

\begin{definition} \label{defn:lin-olig}
Let $V$ be a vector space and let $G$ be a subgroup of $\GL(V)$. We say that $G$ is \emph{linear-oligomorphic} if for any $n \ge 0$ there exists a finite dimensional subspace $U$ of $V$ such that if $W$ is an $n$-dimensional subspace of $V$ then there exists $g \in G$ such that $gW \subset U$.
\end{definition}

The following is our main result connected to these objects:

\begin{theorem} \label{prop:homo-oli}
Let $(V, \omega)$ be a weakly homogeneous countable $\ulambda$-space. Then $G=\Aut(V, \omega)$ is linear-oligomorphic.
\end{theorem}

\begin{proof}
Per Theorem~\ref{thm:main1}, choose an algebraic functor $\Phi \colon \cC_{\umu} \to \cC_{\ulambda}$ such that $(V, \omega)=\Phi(V, \eta)$ for some universal homogeneous $\umu$-space $(V, \eta)$. Then $H=\Aut(V, \eta)$ is linear-oligomorphic \cite[Theorem~B]{homoten}, and $H \subset G$, so $G$ is linear-oligomorphic.
\end{proof}

\end{document}